\documentclass[a4paper,12pt]{article}
\usepackage{amsmath,amsthm,amssymb}
\usepackage[matrix,arrow,curve]{xy}

\newtheorem{lemma}{Lemma}
\newtheorem{prop}{Proposition}

\newtheorem{theorem}{Theorem}
\newtheorem{coro}{Corollary}

\theoremstyle{remark}
\newtheorem{rema}{{\rm\bf Remark}}
\newcommand{\Hom}{{\rm Hom}_{\Lambda}}
\newcommand{\Homk}{{{\rm Hom}_k}}
\newcommand{\charr}{{\rm char\,}}
\DeclareMathOperator{\Ker}{Ker}

\DeclareMathOperator{\HH}{HH}
\renewcommand{\Im}{{\rm Im\,}}
\renewcommand{\Bar}{{\rm Bar}}

\newcommand{\la}{\langle}
\newcommand{\ra}{\rangle}
\renewcommand{\le}{\leqslant}
\renewcommand{\ge}{\geqslant}
\newcommand{\Nod}{\text{\rm CMD}}

\newcommand{\ee}{\varepsilon}
\newcommand{\si}{\sigma}
\newcommand{\ph}{\varphi}

\numberwithin{equation}{section}

\begin{document}
\title{BV-differential on Hochschild cohomology of Frobenius algebras.}
\author{Y. V. Volkov\footnote{The author was supported by RFBR (13-01-00902 A and 14-01-31084 mol\_a).}}
\date{}
\maketitle
\begin{abstract}
For a finite-dimensional Frobenius $k$-algebra $R$ with the Nakayama automorphism $\nu$ we define an algebra $\HH^*(R)^{\nu\uparrow}$.
If the order of $\nu$ is not divisible by the characteristic of $k$, this algebra is isomorphic to the Hochschild cohomology algebra of $R$.. We prove that this algebra is a BV-algebra.  We use this fact to calculate the Gerstenhaber algebra structure and BV-structure on the Hochschild cohomology algebras of a family of self-injective algebras of tree type $D_n$.
\end{abstract}

\section{Introduction}
Hochschild cohomology is a subtle invariant of an associative algebra which carries a lot of information about its structure.
The cohomology theory of associative algebras was introduced by Hochschild. It was later shown in \cite{Gerstenhaber}, that the Hochschild cohomology algebra is a Gerstenhaber algebra. Sometimes we can define a BV-differential on the Hochschild cohomology algebra in such way that the structure of the Gerstenhaber algebra can be defined using this differential. In \cite{Tradler} Tradler shows that such differential can be defined for symmetric algebras. It is proved in \cite{EuS} that a BV-differential can be constructed if the given algebra is a Frobenius algebra with finite stable Calabi-Yau dimension and periodic Hochschild cohomology.
Moreover it is shown in \cite{Ginz} that a BV-differential exists in the case of Calabi-Yau algebra.

Let $R$ be a Frobenius algebra over an algebraically closed field. Let $\nu$ be its Nakayama automorphism. In this paper we define an algebra  $\HH^*(R)^{\nu\uparrow}$. It turns out that $\HH^*(R)^{\nu\uparrow}\simeq\HH^*(R)$ in many cases. The Lie bracket on the algebra $\HH^*(R)$ induces a bracket on $\HH^*(R)^{\nu\uparrow}$. We slightly modify the Tradler's proof for symmetric algebras to construct the BV-differential on the algebra $\HH^*(R)^{\nu\uparrow}$.

In section 5 we apply the obtained results to describe the Gerstenhaber algebra structure and the BV-differential for a family of self-injective algebras of tree type $D_n$. The Hochschild cohomology ring for these algebras was described in terms of generators and relations in \cite{Volk}. Note that the calculation of these structures is a difficult task. There are only a few examples of such calculations. The structure of Gerstenhaber algebra was described for serial self-injective algebras in \cite{XuZhao}. The calculus structure which includes the Gerstenhaber algebra structure was calculated for preprojective algebras of type $A_n$, $D_n$, $E_n$ and $L_n$ in \cite{Eu1} and \cite{Eu2}. Furtermore the Gerstenhaber algebra structure and the BV-differential for the group algebra of the quaternion group of order 8 over a field of characteristic 2 were calculated in \cite{IIVZ}.

\section{Basic definitions and constructions}
Throughout the paper we suppose that $k$ is an algebraically closed field, $R$ is a finite-dimensional $k$-algebra, $\Lambda=R\otimes R^{\rm op}$ is the enveloping algebra of $R$ (we write $\otimes$ instead of $\otimes_k$). The bar-resolution
$$\Bar_*(R):(R\stackrel{\mu}\leftarrow) R^{\otimes 2}\stackrel{d_0}\leftarrow R^{\otimes 3}\leftarrow\cdots\leftarrow R^{\otimes (n+2)}\stackrel{d_n}\leftarrow R^{\otimes (n+3)}\leftarrow\cdots$$
of the algebra $R$ is defined in the following way: $\Bar_n(R)=R^{\otimes(n+2)}$ ($n\ge 0$), $\mu:R\otimes R\rightarrow R$ is the multiplication of $R$ and $d_n$ ($n\ge 0$) is defined by the formula
 $$
 d_n(a_0\otimes\cdots\otimes a_{n+2})=\sum\limits_{i=0}^{n+1}(-1)^ia_0\otimes\cdots\otimes a_{i-1}\otimes a_ia_{i+1}\otimes a_{i+2}\otimes\cdots\otimes a_{n+2},
 $$
where $a_i\in R$ ($0\le i\le n+2$). The homology of the complex $C^*(R)=\Hom(\Bar_*(R),R)$ is called Hochschild cohomology of the algebra $R$.
Note that $C^0(R)\simeq R$ and $C^n(R)=\Hom(R^{\otimes(n+2)},R)\simeq\Homk(R^{\otimes n},R)$. Let introduce notation
  $$
  \begin{aligned}
 &\delta_n^i(f)(a_1\otimes\dots\otimes a_{n+1})\\
 :=&\begin{cases}
 a_1f(a_2\otimes\dots\otimes a_{n+1}),&\mbox{if $i=0$},\\
  (-1)^if(a_1\otimes \cdots \otimes a_ia_{i+1} \otimes \cdots\otimes
a_{n+1}),&\mbox{if $1\le i\le n$},\\
(-1)^{n+1}f(a_1\otimes \cdots \otimes a_n) a_{n+1},&\mbox{if $i=n+1$}.
 \end{cases}
 \end{aligned}
 $$
for $f\in C^n(R)$. Then $C^*(R)$ has the form
 $$
 R\stackrel{\delta_0}\rightarrow\Homk(R,R)\rightarrow\cdots\rightarrow \Homk(R^{\otimes n},R)\stackrel{\delta_n}\rightarrow\Homk(R^{\otimes (n+1)},R)\rightarrow\cdots,
 $$
where $\delta_n=\sum\limits_{i=0}^{n+1}\delta_n^i$.

The cup product $f\smile g \in
C^{n+m}(R)=\mathrm{Hom}_k(R^{\otimes (n+m)}, R)$ for $f\in
C^n(R)$  and $g\in C^m(R)$ is given by
$$
(f\smile g)(a_1\otimes \cdots\otimes
a_{n+m}):=f(a_1\otimes\cdots\otimes a_n)\cdot
g(a_{n+1}\otimes\cdots\otimes a_{n+m}).
$$
This cup product induces a well-defined product on Hochschild
cohomology
$$
\smile \colon \HH^n(R) \times \HH^m(R) \longrightarrow \HH^{n+m}(R)
$$
which turns the graded $k$-vector space $\HH^*(R)=\bigoplus_{n\geq
0}\HH^n(R)$ into a graded commutative algebra (\cite[Corollary
1]{Gerstenhaber}).

Let now define the Lie bracket. Let $f \in C^n(R)$,
$g \in C^m(R)$. If $n, m\geq 1$, then for $1\leq i\leq n$, we set
$$
\begin{aligned}
&(f\circ_i g)(a_1\otimes \cdots\otimes a_{n+m-1})\\
:=&f(a_1\otimes \cdots \otimes a_{i-1}\otimes
g(a_i\otimes \cdots \otimes a_{i+m-1})\otimes a_{i+m}\otimes
\cdots \otimes a_{n+m-1});
\end{aligned}
$$
if $ n\geq 1$ and $m=0$, then $g\in R$ and for $1\leq i\leq
n$, we set
$$
(f\circ_i g)(a_1\otimes \cdots\otimes
a_{n-1}):=f(a_1\otimes \cdots \otimes a_{i-1}\otimes g
\otimes a_{i }\otimes \cdots \otimes a_{n-1}).
$$
Now define
$$
f\circ g
:=\sum_{i=1}^n(-1)^{(m-1)(i-1)}f\circ_i g
$$
and
$$
[f,\, g] :=f\circ g-(-1)^{(n-1)(m-1)}g\circ f.
$$
Note that $[f,\, g]\in C^{n+m-1}(R)$. Then $[\
\,,\,\ ]$ induces a well-defined Lie bracket on Hochschild
cohomology
$$
[\ \,,\,\ ]: \HH^n(R) \times \HH^m(R) \longrightarrow
\HH^{n+m-1}(R)
$$
such that $(\HH^*(R),\, \smile,\, [\ \,,\,\ ])$ is a Gerstenhaber
algebra (\cite{Gerstenhaber}).

A Batalin--Vilkovisky  algebra (BV-algebra for short) is a
Gerstenhaber algebra $(A^\bullet,\, \smile,\, [\ \,,\,\ ])$
together with an operator $\Delta\colon A^\bullet \rightarrow
A^{\bullet-1}$ of degree $-1$ such that $\Delta\circ  \Delta=0$
and
\begin{equation}\label{GerBVeq}
[a,\, b]=-(-1)^{(|a|-1)|b|}(\Delta(a\smile b)- \Delta(a)\smile
b-(-1)^{|a|}a\smile \Delta(b))
\end{equation}
for   homogeneous  elements $a, b\in A^\bullet$. The following Theorem is proved in \cite{Tradler}.

\begin{theorem} \cite[Theorem 1]{Tradler} \label{Tradler}
Let $R$ be a symmetric algebra, i.e. an algebra with a nondegenerate associative symmetric bilinear form $\langle\ \,,\,\ \rangle:R\times R\rightarrow k$.
 For $f\in
C^n(R)=\mathrm{Hom}_k(R^{\otimes n}, R)$ define $\Delta(f) \in
C^{n-1}(R)=\mathrm{Hom}_k(R^{\otimes(n-1)}, R)$ by the formula
$$
\begin{aligned}
&\langle \Delta(f)(a_1\otimes\cdots\otimes a_{n-1}),\;a_n\rangle\\
=&\sum_{i=1}^n  (-1)^{i(n-1)} \langle f(a_i\otimes
\cdots\otimes a_n\otimes a_1\otimes \cdots\otimes
a_{i-1}), \; 1\rangle,
\end{aligned}
$$
where $a_i\in R$ ($1\le i\le n$). The map $\Delta$ induces the differential $\Delta:\HH^n(R) \rightarrow \HH^{n-1}(R)$. Then $(\HH^*(R),\, \smile,\, [\ \,,\,\ ],\, \Delta)$ is a BV-algebra.
\end{theorem}

If $\si:R\rightarrow R$ is an automorphism of the algebra $R$, then we can define a map $\phi_{\si}:C^n(R)\rightarrow C^n(R)$ by the formula
$$
\big(\phi_{\si}(f)\big)(a_1\otimes\dots\otimes a_n)=\si^{-1}(f(\si(a_1)\otimes\dots\otimes\si(a_n))),
$$
where $f\in C^n(R)=\mathrm{Hom}_k(R^{\otimes n}, R)$, $a_i\in R$ ($1\le i\le n$). From here on we write $f^{\si}$ instead of $\phi_{\si}(f)$. It is easy to show that $(\delta_nf)^{\si}=\delta_nf^{\si}$ for $f\in C^n(R)$. Consequently, the map $\phi$ induces a map on Hochschild cohomology
$(\,\,)^\si:\HH^n(R) \rightarrow \HH^{n}(R)$. For a map $\si:X\rightarrow X$ we denote by $X^{\si}$ the set $\{x\in X\mid \si(x)=x\}$. Then $\delta_n:C^n(R)\rightarrow C^{n+1}(R)$ induces the map $\delta_n^{\si}:C^n(R)^{\si}\rightarrow C^{n+1}(R)^{\si}$. We denote by $\HH^n(R)^{\si\uparrow}$ the homology of the complex $(C^n(R)^{\si},\delta_n^{\si})$ and define
$$
\HH^*(R)^{\si\uparrow}:=\bigoplus\limits_{n\ge 0}\HH^n(R)^{\si\uparrow}=\bigoplus\limits_{n\ge 0}\Ker(\delta_n^{\si})/\Im(\delta_{n-1}^{\si}).
$$

It is proved in \cite{Volk_new} that $\smile:C^n(R) \times C^m(R) \longrightarrow C^{n+m}(R)$ defined above determines an algebra structure on $\HH^*(R)^{\si\uparrow}$ and the inclusion of $C^n(R)^{\si}$ into $C^n(R)$ induces an algebra homomorphism
$$
\Theta_R^{\si}:\HH^*(R)^{\si\uparrow}\rightarrow \HH^*(R)^{\si}.
$$
Moreover, $\Theta_R^{\si}$ is bijective if ${\rm ord}(\si)<\infty$ and $\charr k{\not|}{\rm ord}(\si)$. The fact that $[\,,\,]:C^n(R) \times C^m(R) \longrightarrow C^{n+m-1}(R)$ induces a Gerstenhaber algebra structure on $\HH^*(R)^{\si\uparrow}$ can be proved analogously to the fact that $[\,,\,]$ induces a Gerstenhaber algebra structure on $\HH^*(R)$ (cf. \cite{Gerstenhaber}). It is easy to see that $\Theta_R^{\si}$ is a homomorphism of Gerstenhaber algebras.

If $d:M\rightarrow  N$ is a morphism of $R$-bimodules and $\si_1$, $\si_2$ are automorphisms of the algebra $R$, then we denote by $d_{(\si_1,\si_2)}:{}_{\si_1}M_{\si_2}\rightarrow{}_{\si_1}N_{\si_2}$ the morphism of $R$-bimodules, which is determined by the formula $d_{(\si_1,\si_2)}(m)=d(m)$ for $m\in M$ (we denote by ${}_{\si_1}M_{\si_2}$ the bimodule which is equal to $M$ as $k$-linear space with multiplication $*$ defined by the equality $a*m*b=\si_1(a)m\si_2(b)$ for $m\in M$, $a,b\in R$).

Algebra $R$ is called a Frobenius algebra if there is a linear map $\epsilon:R\rightarrow k$ such that the bilinear form
$\langle a,b\rangle=\epsilon(ab)$
is nondegenerated. The Nakayama automorphism $\nu:R\rightarrow R$ is the automorphism which satisfies the equation $\langle a,b\rangle=\langle b,\nu(a)\rangle$ for all $a,b\in R$. From here on we assume that $R$ is a Frobenius algebra, $\langle\,,\,\rangle$ is the corresponding bilinear form and $\nu$ is the Nakayama automorphism defined by it.

 \section{The generalization of Tradler's Theorem}
 Let $f\in C^n(R)$, $n\ge 1$. Define $\Delta_if\in C^{n-1}(R)$ by the equation
 $$
 \begin{aligned}
 &\la\Delta_if(a_1\otimes\dots\otimes a_{n-1}),a_n\ra\\
 =&\la f(a_i\otimes\dots\otimes a_{n-1}\otimes a_n\otimes \nu a_1\otimes\dots\otimes \nu a_{i-1}),1\ra,
 \end{aligned}
 $$
where $a_i\in R$ ($1\le i\le n$). Further define
\begin{equation}\label{Delta}
 \Delta:=\sum\limits_{i=1}^n(-1)^{i(n-1)}\Delta_i:C^n(R)\rightarrow C^{n-1}(R).
\end{equation}

\begin{lemma}\label{chainty}
$\delta_{n-1}(\Delta f)+\Delta\delta_n(f)=f^{\nu}-f\hspace{0.2cm}\forall f \in C^n(R)$.
\end{lemma}
\begin{proof} Let $f\in C^n(R)$, $a_i\in R$ ($1\le i\le n+1$). Set $$A_i:=a_i\otimes\dots\otimes a_{n+1}\otimes\nu a_1\otimes\dots\otimes\nu a_{i-1}\in R^{\otimes (n+1)}.$$
Direct calculations show that
$$
\begin{aligned}
&\la (-1)^{i(n-1)}\delta_{n-1}^j(\Delta_i f)(a_1\otimes\dots\otimes a_n), a_{n+1}\ra\\
=&\begin{cases}
\la(-1)^{(i+1)n+1}\delta_n^{n-i+j+1}(f)(A_{i+1}),1\ra,&\mbox{if $0\le j\le i-1$,}\\
\la(-1)^{in+1}\delta_n^{j-i+1}(f)(A_i),1\ra,&\mbox{if $i\le j\le n$.}
\end{cases}
\end{aligned}$$
Adding these equalities for $0\le j\le n$ and $1\le i\le n$, we obtain
$$
\begin{aligned}
&\la \delta_{n-1}(\Delta f)(a_1\otimes\dots\otimes a_n), a_{n+1}\ra+\la \Delta \delta_n(f)(a_1\otimes\dots\otimes a_n), a_{n+1}\ra\\
=&\left\la\sum\limits_{i=1}^{n+1}(-1)^{in+1}\sum\limits_{j=1}^n\delta_n^j(f)(A_i),1\right\ra+
\left\la\sum\limits_{i=1}^{n+1}(-1)^{in}\delta_n(f)(A_i),1\right\ra\\
=&\left\la\sum\limits_{i=1}^{n+1}(-1)^{in}(\delta_n^0(f)(A_i)+\delta_n^{n+1}(f)(A_i)),1\right\ra.
\end{aligned}
$$
Since $\la\delta_n^0(f)(A_i)+(-1)^n\delta_n^{n+1}(f)(A_{i+1}),1\ra=0$
 for $1\le i\le n$, we have
$$
\begin{aligned}
&\la \delta_{n-1}(\Delta f)(a_1\otimes\dots\otimes a_n), a_{n+1}\ra+\la \Delta \delta_n(f)(a_1\otimes\dots\otimes a_n), a_{n+1}\ra\\
=&\la(-1)^n\delta_n^{n+1}(f)(A_1)+\delta_n^0(f)(A_{n+1}),1\ra\\
=&\la a_{n+1}f(\nu a_1\otimes\dots \otimes\nu a_n),1\ra-\la f(a_1\otimes\dots \otimes a_n)a_{n+1},1\ra\\
=&\la (f^{\nu}-f)(a_1\otimes\dots \otimes a_n),a_{n+1}\ra.
\end{aligned}
$$
\end{proof}

The next two corollaries follow directly from Lemma \ref{chainty}.

\begin{coro}
The map $\Delta$ defined by the formula \eqref{Delta} induces a map
$$\Delta:\HH^n(R)^{\nu\uparrow}\rightarrow \HH^{n-1}(R)^{\nu\uparrow}.$$
\end{coro}

\begin{coro}
$\HH^*(R)^{\nu}=\HH^*(R)$.
\end{coro}

For a special case, the second corollary is proved in \cite{GuGu}.
Now we are able to formulate the following generalization of Theorem \ref{Tradler}.

\begin{theorem}\label{Tradler_new}
Let $R$ be a Frobenius algebra with bilinear form $\la\,,\,\ra$ and Nakayama automorphism $\nu$. Then $\Delta$ defined by \eqref{Delta} induces a $BV$-algebra structure on the Gerstenhsber algebra $(\HH^*(R)^{\nu\uparrow},\smile,[\,,\,])$.
\end{theorem}

\begin{proof} Let us prove that $\Delta\circ\Delta=0$ in $\HH^*(R)^{\nu\uparrow}$. It is enough to prove that any element of $\HH^n(R)^{\nu\uparrow}$ can be represented by an element $f\in \Ker\delta_n^{\nu}$ such that $f(a_1\otimes\dots\otimes a_n)=0$ if $a_i=1$ for some $1\le i\le n$. Let us define the maps
$s_n^i:C^{n+1}(R)\rightarrow C^n(R)$ ($0\le i\le n$) by the formula
$$(s_n^i(g))(a_1\otimes \dots\otimes a_n)=(-1)^ig(a_1\otimes \dots\otimes a_i\otimes 1\otimes a_{i+1}\dots\otimes a_n)$$
for $g\in C^{n+1}(R)$. It is clear that $s_n^i$ induces a map $s_n^{i,\nu}:C^{n+1}(R)^{\nu}\rightarrow C^n(R)^{\nu}$.
Now we can prove by induction on $N$ that any element $f'\in\HH^n(R)^{\nu\uparrow}$ can be represented by an element $f\in \Ker\delta_n^{\nu}$ such that $f(a_1\otimes\dots\otimes a_n)=0$ if $a_i=1$ for some $1\le i\le N$. Indeed, if $N=0$, then the required assertion is obvious. Now assume that the cohomology class of $f$ is equal to $f'$ and $f(a_1\otimes\dots\otimes a_n)=0$ if $a_i=1$ for some $1\le i\le N-1$. Then direct calculations show that
$$(f-\delta_{n-1}^{\nu}s_{n-1}^{N-1,\nu}(f))(a_1\otimes\dots\otimes a_n)=(s_{n}^{N-1,\nu}\delta_n^{\nu}(f))(a_1\otimes\dots\otimes a_n)=0$$
 if $a_i=1$ for some $1\le i\le N$.

Now assume that $f\in C^n(R)^{\nu}$, $g\in C^m(R)^{\nu}$. Let us introduce the following notation
$$
\Delta'(f\otimes g):=\sum\limits_{i=1}^m(-1)^{i(n+m-1)}\Delta_i(f\smile g).
$$
Then $\Delta(f\smile g)=\Delta'(f\otimes g)+(-1)^{nm}\Delta'(g\otimes f)$.

The remaining part of the proof is analogous to the proof of Theorem \ref{Tradler} (cf. \cite{Tradler}).
\end{proof}

Note that a Frobenius algebra is symmetric if and only if its Nakayama automorphism equals ${\rm Id}_R$ for some bilinear form. So Theorem \ref{Tradler} is a special case of Theorem \ref{Tradler_new}.

\begin{coro}
If ${\rm ord}(\nu)<\infty$ and $\charr k{\not|}{\rm ord}(\nu)$, then the map defined by the formula \eqref{Delta} induces a BV-differential on
the Gerstenhaber algebra $(\HH^*(R),\smile,[\,,\,])$.
\end{coro}

\begin{proof} This follows from Corollary 2 of Lemma \ref{chainty} and the fact that $\Theta_R^{\nu}:\HH^*(R)^{\nu\uparrow}\rightarrow \HH^*(R)^{\nu}$ is an isomorphism of Gerstenhaber algebras if ${\rm ord}(\nu)<\infty$ and $\charr k{\not|}{\rm ord}(\nu)$.
\end{proof}

\section{On a family of self-injective algebras of tree type $D_n$}

Let $n\ge 4$, $r\ge 1$. Let us define a $k$-algebra $R=R(n,r)$.
Consider a  quiver with relations $(\mathcal Q,I)$. Its set of vertices is $\mathcal Q_0= \mathbb{Z}_r\times \{j\in \mathbb N \mid 1\le j \le n\}$. The set of arrows $\mathcal Q_1$ of the quiver $\mathcal Q$ consists of the following elements:
$$
\begin{aligned}
&\gamma_{i,p}:(i,n-2)\rightarrow (i,p),\:\:\beta_{i,p}:(i,p)\rightarrow (i+1,1),\\
&\alpha_{i,j}: (i,j)\rightarrow (i,j+1)\hspace{0.2cm}(1\le i\le r, p\in\{n-1,n\}, 1\le j\le n-3).
\end{aligned}
$$
In addition we use the following auxiliary notation:
$$
 \tau_i=\beta_{i,n}\gamma_{i,n},\:\:\omega_{i,j_2,j_1}=\alpha_{i,j_2}\dots \alpha_{i,j_1},\:\: \mu_{i,j}=\omega_{i,j,1},\:\:\eta_{i,j}=\omega_{i,n-3,j}.
 $$
From here on for the uniformity of the notation
 we additionally suppose that the empty product of the arrows of the quiver
 is identified with an appropriate idempotent of the algebra $R$;
 for example, $\mu_{i,0}=e_{i,1}$, $\omega_{i,j-1,j}=e_{i,j}$, $\eta_{i,n-2}=e_{i,n-2}$.
Denote by $\phi:\{1,\dots,n\}\rightarrow\{1,\dots,n\}$ the map such that $\phi(j)=j$ for $1\le j\le n-2$, $\phi(n-1)=n$ and $\phi(n)=n-1$.

The ideal $I$ is generated by the elements
$$
\begin{aligned}
&\gamma_{i,\phi(p)}\eta_{i,1}\beta_{i-1,p},\hspace{0.1cm}\beta_{i,n-1}\gamma_{i,n-1}-\tau_{i},\hspace{0.1cm}\mu_{i+1,j}\tau_i\eta_{i,j}\\
&(1\le i\le r, 1\le j\le n-3, p\in\{n-1, n\}).
\end{aligned}
 $$
$$
  \xymatrix@=0.45cm
  {
  &&&&\txt{\scriptsize $i,n-1$}\ar@{->}[dl]_*\txt{\tiny $\beta_{i,n-1}$}\\
  &&\txt{\scriptsize i+1,2} \ar@{->}[dll]_*\txt{\tiny $\alpha_{i+1,2}$}&\txt{\scriptsize i+1,1}\ar@{->}[l]_*\txt{\tiny $\alpha_{i+1,1}$}&&
  \txt{\scriptsize i,n-2}\ar@{->}[ul]_*\txt{\tiny $\gamma_{i,n-1}$}\ar@{->}[dl]^*\txt{\tiny $\gamma_{i,n}$}&
  \txt{\scriptsize i,n-3}\ar@{->}[l]_*\txt{\tiny $\alpha_{i,n-3}$}&\\
  &&&&\txt{\scriptsize $i,n$}\ar@{->}[ul]^*\txt{\tiny $\beta_{i,n}$}&&&&\ar@{->}[ull]_*\txt{\tiny $\alpha_{i,n-4}$}\\
  \\
  \ar@{.}[uu]\ar@{->}[dr]_*\txt{\tiny $\alpha_{r,n-3}$}&&\txt{\scriptsize $r,n-1$}\ar@{->}[dr]^*\txt{\tiny $\beta_{r,n-1}$}&&&&
  \txt{\scriptsize $1,n-1$}\ar@{->}[dr]^*\txt{\tiny $\beta_{1,n-1}$}&&\ar@{.}[uu]\\
  &\txt{\scriptsize r,n-2}\ar@{->}[ur]^*\txt{\tiny $\gamma_{r,n-1}$}\ar@{->}[dr]_*\txt{\tiny $\gamma_{r,n}$}&&
  \txt{\scriptsize 1,1}\ar@{->}[r]_*\txt{\tiny $\alpha_{1,1}$}&\cdots\ar@{->}[r]_*\txt{\tiny $\alpha_{1,n-3}$}&
  \txt{\scriptsize 1,n-2}\ar@{->}[ur]^*\txt{\tiny $\gamma_{1,n-1}$}\ar@{->}[dr]_*\txt{\tiny $\gamma_{1,n}$}&&
  \txt{\scriptsize 2,1}\ar@{->}[ur]_*\txt{\tiny $\alpha_{2,1}$}\\
  &&\txt{\scriptsize $r,n$}\ar@{->}[ur]_*\txt{\tiny $\beta_{r,n}$}&&&&\txt{\scriptsize $1,n$}\ar@{->}[ur]_*\txt{\tiny $\beta_{1,n}$}
  }
$$
Set $R=R(n,r)=k\mathcal Q/I$. It is easily verified that $R$ is a Frobenius algebra. For two paths $w_1$ and $w_2$ of the quiver $\mathcal Q$ the bilinear form $\la\,,\,\ra$ is defined in the following way:
$$
\la w_1,w_2\ra=\epsilon(w_1w_2):=
\begin{cases}
1,&\mbox{if $w_1w_2$ is a nonzero path of length $n-1$,}\\
0&\mbox{overwise.}
\end{cases}
$$
This bilinear form determines the Nakayama automorphism $\nu$. This automorphism is defined on idempotents and arrows by the formulas
$$
\begin{aligned}
&\nu(e_{i,t})=e_{i-1,t},\:\:\nu(\alpha_{i,j})=\alpha_{i-1,j},\:\:\nu(\gamma_{i,p})=\gamma_{i-1,p},\:\:\nu(\beta_{i,p})=\beta_{i-1,p}\\
&(1\le i\le r,1\le t\le n, 1\le j\le n-2, p\in\{n-1,n\}).
\end{aligned}
$$
It is clear that the set
$$
 \begin{aligned}
B_R=&\{\mu_{i+1,t-1}\tau_i\eta_{i,j}\mid 1\le
i\le r,1\le t\le j\le n-2\}\\
&\cup\{\omega_{i,t-1,j}\mid 1\le i\le r,1\le j\le
t\le n-2\}\\
 & \cup\{\gamma_{i,p}\eta_{i,j}\mid 1\le i\le r, 1\le j\le n-2, p\in\{n-1,n\}\}\\
& \cup\{\mu_{i+1,j-1}\beta_{i,p}\mid 1\le i\le r, 1\le j\le n-2, p\in\{n-1,n\}\}\\
 &
\cup\{\gamma_{i,p}\mu_{i,n-3}\gamma_{i-1,p}\mid 1\le
 i\le r, p\in\{n-1,n\}\}
 \end{aligned}
$$
is a $k$-basis of $R$. We define for $b\in B_R$ the element $\bar b\in B_R$ by the equalities
$$
 \begin{aligned}
\overline{\mu_{i+1,t-1}\tau_i\eta_{i,j}}=&\omega_{i+1,j-1,t}\hspace{0.2cm}(1\le i\le r,1\le t\le j\le n-2)\\
\overline{\omega_{i,t-1,j}}=&\mu_{i+1,j-1}\tau_{i}\eta_{i,t}\hspace{0.2cm}(1\le i\le r,1\le j\le t\le n-2)\\
\overline{\gamma_{i,p}\eta_{i,j}}=&\mu_{i+1,j-1}\beta_{i,p}\hspace{0.2cm}(1\le i\le r, 1\le j\le n-2, p\in\{n-1,n\})\\
\overline{\mu_{i+1,j-1}\beta_{i,n-1}}=&\gamma_{i+1,p}\eta_{i+1,j}\hspace{0.2cm}(1\le i\le r, 1\le j\le n-2, p\in\{n-1,n\})\\
\overline{\gamma_{i,p}\mu_{i,n-3}\gamma_{i-1,p}}=&e_{i,p}\hspace{0.2cm}(1\le i\le r, p\in\{n-1,n\}).\\
 \end{aligned}
$$
It is easy to show that $\la a,b\ra=\begin{cases}1,&\mbox{if $b=\bar a$,}\\0,&\mbox{if $b\not=\bar a$,}\end{cases}$ for $a,b\in B_R$.

The Hochschild cohomology algebra of $R(n,r)$ is described in \cite{VolkGen} and \cite{Volk}.
Let us recall some results of these works. Denote by $e_x$ the idempotent of the algebra $R$ corresponding to a vertex $x$ of the quiver $\mathcal Q$. Then $\{e_x\otimes e_y\}_{x,y}$ is a full set of orthogonal primitive idempotents for the algebra $\Lambda$. Denote by $P_{[x][y]}=\Lambda e_{x}\otimes e_{y}$ the projective $\Lambda$-module, which corresponds to idempotent $e_{x}\otimes
e_{y}$. Let $\si:R\rightarrow R$ be the automorphism of $R$, which is defined on the idempotents and arrows by the formulas
$$
\begin{aligned}
&\sigma(e_{i,j})=e_{i+n-1,\phi^{n}(j)},\sigma(\alpha_{i,j})=\alpha_{i+n-1,j},\sigma(\gamma_{i,p})=-\gamma_{i+n-1,\phi^{n}(p)},\\
&\sigma(\beta_{i,p})=\beta_{i+n-1,\phi^{n}(p)}\hspace{0.2cm}(1\le i\le r, 1\le j\le n-2, p\in\{n-1,n\}).
\end{aligned}$$

Denote by $Q_t$ the $t$-th module in the minimal projective bimodule resolution of $R$. Then
$$
\begin{aligned}
Q_{2m}&=  \bigoplus\limits_{i=1}^r
 \bigg(
    \Big(
        {\:\bigoplus\limits_{j=1}^{n-2-m}P_{[i+m,j+m][i,j]}}
    \Big)\\
  &\oplus
   \Big(
  {\bigoplus\limits_{j=n-1-m}^{n-2}P_{[i+m,j+m-(n-2)][i,j]}
  }
  \Big)\\
 & \oplus
  P_{[i+m,\phi^m(n-1)][i,n-1]} \oplus
 P_{[i+m,\phi^m(n)][i,n]}
  \bigg)\hspace{0.5cm}
(0\le m\le n-2),\\
 Q_{2m+1}&=\bigoplus\limits_{i=1}^r
\bigg(\Big(\bigoplus\limits_{j=1}^{n-3-m}P_{[i+m,j+m+1][i,j]}
\Big)\oplus
P_{[i+m,n-1][i,n-2-m]}\notag\\
 &\oplus
P_{[i+m,n][i,n-2-m]}\oplus
\Big(\bigoplus\limits_{j=n-1-m}^{n-2}P_{[i+m+1,j+m-(n-2)][i,j]}
 \Big)\notag\\
& \oplus P_{[i+m+1,m+1][i,n-1]}\oplus
 P_{[i+m+1,m+1][i,n]}\bigg)\hspace{0.5cm}(0\le m\le n-3).
\end{aligned}
$$
Moreover, $Q_{t+l(2n-3)}={}_{\si^l}(Q_t)_1$ for $0\le t\le 2n-4$, $l>0$. The definitions of the differentials $d_t^Q:Q_{t+1}\rightarrow Q_t$ can be found in \cite{VolkGen}. The augmentation map $\mu:Q_0\rightarrow R$ is defined by the formula $\mu(a\otimes b)=ab\in R$ for $a\otimes b\in P_{[x][x]}$ ($x\in \mathcal Q_0$). It is easy to check that ${}_{\nu}(P_{[i_1,j_1][i_2,j_2]})_{\nu}\simeq P_{[i_1+1,j_1][i_2+1,j_2]}$. These isomorphisms give isomorphisms $\theta_t:Q_t\rightarrow {}_{\nu}(Q_t)_{\nu}$ for $t\ge 0$. Using the description of the differentials $d_t^Q$ it is easy to verify that $\theta_*$ is a chain map, which lifts the isomorphism $\phi_{\nu}:R\rightarrow {}_{\nu}R_{\nu}$ defined by the formula $\phi_{\nu}(a)=\nu a$ for $a\in R$. Let $f\in\Hom(Q_t,R)$, $fd_t^Q=0$ and $f$ is defined on the direct summands of $Q_t$ by the equalities $f(e_{[i_1,j_1][i_2,j_2]})=f_{[i_1,j_1][i_2,j_2]}\in R$. It follows from \cite[Lemma 2]{Volk_new} that $f^{\nu}$ is defined by the equalities
$$
f^{\nu}(e_{[i_1,j_1][i_2,j_2]})=\nu^{-1}f_{[i_1-1,j_1][i_2-1,j_2]}.
$$

To calculate the Lie bracket we will partially use the algorithm described in \cite{IIVZ}. To apply this algorithm we need the homomorphism of left $R$-modules $D_t:Q_t\rightarrow Q_{t+1}$ ($t\ge 0$), $D_{-1}:R\rightarrow Q_0$, which satisfy the equations
\begin{equation}\label{weak_homot}
\begin{aligned}
D_{t+1}D_t&=0\hspace{0.2cm}(t\ge 0),\:\:D_{t-1}d_{t-1}^Q+d_t^QD_t={\rm Id}_{Q_t}\hspace{0.2cm}(t\ge 1),\\
D_{-1}\mu+d_0^QD_0&={\rm Id}_{Q_0},\:\:\mu D_{-1}={\rm Id}_R.
\end{aligned}
\end{equation}
Let $D_{-1}:R\rightarrow Q_0$ be a homomorphism of left modules such that $D_{-1}(e_x)=e_{[x][x]}$ for all $x\in \mathcal Q_0$.

Now let us define the homomorphisms $D_t$ ($t\ge 0$). We will define them on the elements of the form $e_x\otimes b$ ($x\in\mathcal{Q}_0$, $b\in B_R$).  Let $0\le m\le n-3$. Then\\
$D_{2m}(e_{i+m,j+m}\otimes e_{i,j}b)$ $(1\le i\le r,1\le j\le n-2-m)$ is equal to\\
$$\sum\limits_{s=q}^{j-1}\omega_{i+m,j+m-1,s+m+1}\otimes\omega_{i,s-1,q}$$
if $b=\omega_{i,j-1,q}$, $1\le q\le j$,\\
$$\sum\limits_{s=1}^{j-1}\omega_{i+m,j+m-1,s+m+1}\otimes\mu_{i,s-1}\beta_{i-1,p}
+\omega_{i+m,j+m-1,m+1}\otimes e_{i-1,p}$$
if $b=\mu_{i,j-1}\beta_{i-1,p}$, $p\in\{n-1,n\}$,\\
$$\begin{aligned}
&\sum\limits_{s=1}^{j-1}\omega_{i+m,j+m-1,s+m+1}\otimes\mu_{i,s-1}\tau_{i-1}\eta_{i-1,q}\\
+&\omega_{i+m,j+m-1,m+1}\otimes \gamma_{i-1,n-1}\eta_{i-1,q}-\omega_{i+m,j+m-1,q+m-(n-2)}\otimes e_{i-1,q}
\end{aligned}$$ if $b=\mu_{i,j-1}\tau_{i-1}\eta_{i-1,q}$, $n-1-m\le q\le n-2$,\\
$$
\begin{aligned}
&\sum\limits_{s=1}^{j-1}\omega_{i+m,j+m-1,s+m+1}\otimes\mu_{i,s-1}\tau_{i-1}\eta_{i-1,q}\\
+&\omega_{i+m,j+m-1,m+1}\otimes \gamma_{i-1,n-1}\eta_{i-1,q}\\
+&\mu_{i+m,j+m-1}\beta_{i+m-1,\phi^m(n-1)}\otimes \omega_{i-1,n-3-m,q}\\
+&\sum\limits_{s=q}^{n-3-m}\mu_{i+m,j+m-1}\tau_{i+m-1}\eta_{i+m-1,s+m+1}\otimes \omega_{i-1,s-1,q}
\end{aligned}$$
if $b=\mu_{i,j-1}\tau_{i-1}\eta_{i-1,q}$, $j\le q\le n-2-m$;\\
$D_{2m}(e_{i+m,j+m-(n-2)}\otimes e_{i,j}b)$ $(1\le i\le r,n-1-m\le j\le n-2)$ is equal to 0 if $b\not=\mu_{i,j-1}\tau_{i-1}\eta_{i-1,j}$, and is equal to $-e_{i+m,j+m-(n-2)}\otimes e_{i-1,j}$ if $b=\mu_{i,j-1}\tau_{i-1}\eta_{i-1,j}$;\\
$D_{2m}(e_{i+m,\phi^m(p)}\otimes e_{i,p}b)$ $(1\le i\le r,p\in\{n-1,n\})$ is equal to 0\\
if $b=e_{i,p}$ or $b=\gamma_{i,p}\eta_{i,q}$, $n-1-m\le q\le n-2$,\\
$$
e_{i+m,\phi^m(p)}\otimes \omega_{i,n-3-m,q}+\sum\limits_{s=q}^{n-3-m}\gamma_{i+m,\phi^m(p)}\eta_{i+m,s+m+1}\otimes\omega_{i,s-1,q}
$$
if $b=\gamma_{i,p}\eta_{i,q}$, $1\le q\le n-2-m$,\\
$$
\begin{aligned}
&e_{i+m,\phi^m(p)}\otimes \mu_{i,n-3-m}\beta_{i-1,p}\\
+&\sum\limits_{s=q}^{n-3-m}\gamma_{i+m,\phi^m(p)}\eta_{i+m,s+m+1}\otimes\mu_{i,s-1,q}\beta_{i-1,p}\\
+&\gamma_{i+m,\phi^m(p)}\eta_{i+m,m+1}\otimes e_{i-1,p}
\end{aligned}
$$
if $b=\gamma_{i,p}\eta_{i,1}\beta_{i-1,p}$.

Let $0\le m\le n-3$. Then\\
$D_{2m+1}(e_{i+m,j+m+1}\otimes e_{i,j}b)$ $(1\le i\le r,1\le j\le n-3-m)$ is equal to 0 if $b\not=\mu_{i,j-1}\tau_{i-1}\eta_{i-1,j}$, and is equal to $e_{i+m,j+m+1}\otimes e_{i-1,j}$ if $b=\mu_{i,j-1}\tau_{i-1}\eta_{i-1,j}$;\\
$D_{2m+1}(e_{i+m+1,j+m-(n-2)}\otimes e_{i,j}b)$ $(1\le i\le r,n-1-m\le j\le n-2)$ is equal to\\
$$\sum\limits_{s=q}^{j-1}\omega_{i+m+1,j+m-(n-1),s+m+(n-3)}\otimes\omega_{i,s-1,q}$$
if $b=\omega_{i,j-1,q}$, $n-1-m\le q\le j$,\\
$$\sum\limits_{s=n-2-m}^{j-1}\omega_{i+m+1,j+m-(n-1),s+m+(n-3)}\otimes\omega_{i,s-1,q}$$
if $b=\omega_{i,j-1,q}$, $1\le q\le n-2-m$,\\
$$\begin{aligned}
&\sum\limits_{s=n-2-m}^{j-1}\omega_{i+m+1,j+m-(n-1),s+m+(n-3)}\otimes\mu_{i,s-1}\beta_{i-1,n-1}\\
+&\mu_{i+m+1,j+m-(n-1)}\beta_{i+m,\phi^{m+1}(n-1)}\otimes e_{i-1,n-1}
\end{aligned}$$
if $b=\mu_{i,j-1}\beta_{i-1,n-1}$,\\
$$
\begin{aligned}
&\sum\limits_{s=n-2-m}^{j-1}\omega_{i+m+1,j+m-(n-1),s+m+(n-3)}\otimes\mu_{i,s-1}\beta_{i-1,n}\\
-&\mu_{i+m+1,j+m-(n-1)}\beta_{i+m,\phi^{m+1}(n)}\otimes e_{i-1,n}
\end{aligned}$$
if $b=\mu_{i,j-1}\beta_{i-1,n}$,\\
$$
\begin{aligned}
&\sum\limits_{s=n-2-m}^{j-1}\omega_{i+m+1,j+m-(n-1),s+m+(n-3)}\otimes\mu_{i,s-1}\tau_{i-1}\eta_{i-1,q}\\
-&\mu_{i+m+1,j+m-(n-1)}\beta_{i+m,\phi^{m+1}(n)}\otimes \gamma_{i-1,n}\eta_{i-1,q}\\
+&\mu_{i+m+1,j+m-(n-1)}\beta_{i+m,\phi^{m+1}(n-1)}\otimes \gamma_{i-1,n-1}\eta_{i-1,q}\\
-&\sum\limits_{s=q}^{n-2}\mu_{i+m+1,j+m-(n-1)}\tau_{i+m}\eta_{i+m,s+m-(n-3)}\otimes\omega_{i-1,s-1,q}
\end{aligned}$$
if $b=\mu_{i,j-1}\tau_{i-1}\eta_{i,q}$ ($j\le q\le n-2$);\\
$D_{2m+1}(e_{i+m+1,m+1}\otimes e_{i,n-1}b)$ $(1\le i\le r)$ is equal to 0\\
if $b\not=\gamma_{i,n-1}\eta_{i,1}\beta_{i-1,n-1}$, and is equal to
$$\begin{aligned}
&\sum\limits_{s=n-2-m}^{n-2}\omega_{i+m+1,m,s+m-(n-1)}\otimes\mu_{i,s-1}\beta_{i-1,n-1}\\
+&\mu_{i+m+1,m}\beta_{i+m,\phi^{m+1}(n-1)}\otimes e_{i-1,n-1}
\end{aligned}$$
if $b=\gamma_{i,n-1}\eta_{i,1}\beta_{i-1,n-1}$;\\
$D_{2m+1}(e_{i+m+1,m+1}\otimes e_{i,n}b)$ $(1\le i\le r)$ is equal to 0\\
if $b=e_{i,n}$,\\
$$-\sum\limits_{s=q}^{n-2}\omega_{i+m+1,m,s+m+(n-3)}\otimes\omega_{i,s-1,q}$$
if $b=\gamma_{i,n}\eta_{i,q}$, $n-1-m\le q\le n-2$,\\
$$-\sum\limits_{s=n-2-m}^{n-2}\omega_{i+m+1,m,s+m+(n-3)}\otimes\omega_{i,s-1,q}$$
if $b=\gamma_{i,n}\eta_{i,q}$, $1\le q\le n-2-m$,\\
$$
\begin{aligned}
&-\sum\limits_{s=n-2-m}^{n-2}\omega_{i+m+1,m,s+m-(n-1)}\otimes\mu_{i,s-1}\beta_{i-1,n-1}\\
&+\mu_{i+m+1,m}\beta_{i+m,\phi^{m+1}(n)}\otimes e_{i-1,n}
\end{aligned}
$$
if $b=\gamma_{i,n}\eta_{i,1}\beta_{i-1,n}$;\\
$D_{2m+1}(e_{i+m,\phi^m(p)}\otimes e_{i,n-2-m}b)$ $(1\le i\le r,p\in\{n-1,n\})$ is equal to 0\\
if $b=\omega_{i,n-3-m,q}$, $1\le q\le n-2-m$ or $b=\mu_{i,n-3-m}\beta_{i-1,p}$,\\
$$e_{i+m\phi^m(p)}\otimes e_{i-1,\phi(p)}$$
if $b=\mu_{i,n-3-m}\beta_{i-1,\phi(p)}$,\\
$$e_{i+m\phi^m(n)}\otimes \gamma_{i-1,n-1}\eta_{i,q}$$
if $p=n$, $b=\mu_{i,n-3-m}\tau_{i-1}\eta_{i-1,q}$ ($n-2-m\le q\le n-2$),\\
$$e_{i+m\phi^m(n-1)}\otimes \gamma_{i-1,n}\eta_{i,q}+\sum\limits_{s=q}^{n-2}\gamma_{i+m\phi^m(n-1)}\eta_{i+m,s+m-(n-3)}\otimes\omega_{i-1,s-1,q}$$
if $p=n-1$, $b=\mu_{i,n-3-m}\tau_{i-1}\eta_{i-1,q}$ ($n-2-m\le q\le n-2$).

We define the homomorphism $D_{2n-4}$ in the following way.\\
$D_{2n-4}(e_{i+n-2,j}\otimes e_{i,j}b)$ is equal to 0 if $b\not=\mu_{i,j-1}\tau_{i-1}\eta_{i,j}$, and is equal to $-e_{i+n-2,j}\otimes e_{i-1,j}$ if $b=\mu_{i,j-1}\tau_{i-1}\eta_{i,j}$;\\
$D_{2n-4}(e_{i+n-2,\phi^{n-2}(n-1)}\otimes e_{i,n-1}b)$ is equal to 0 if $b\not=\gamma_{i,n-1}\eta_{i-1,1}\beta_{i-1,n-1}$, and is equal to $-e_{i+n-2,\phi^{n-2}(n-1)}\otimes e_{i-1,n-1}$ if $b=\gamma_{i,n-1}\eta_{i-1,1}\beta_{i-1,n-1}$;\\
$D_{2n-4}(e_{i+n-2,\phi^{n-2}(n)}\otimes e_{i,n}b)$ is equal to 0 if $b\not=\gamma_{i,n}\eta_{i-1,1}\beta_{i-1,n}$, and is equal to $e_{i+n-2,\phi^{n-2}(n)}\otimes e_{i-1,n}$ if $b=\gamma_{i,n}\eta_{i-1,1}\beta_{i-1,n}$.

Moreover, $D_{t+l(2n-3)}=(D_t)_{(\si^l,1)}$ for $0\le t\le 2n-4$, $l>0$. Direct calculations show that the maps $D_t$ ($t\ge -1$) satisfy the equalities \eqref{weak_homot}.

Let $\theta_{\Bar,t}:\Bar_n(R)\rightarrow{}_{\nu}(\Bar_n(R))_{\nu}$ be defined by the formula $\theta_{\Bar,t}(a_0\otimes\dots\otimes a_{t+1})=\nu a_0\otimes\dots\otimes\nu a_{t+1}$.
It is easy to show that the chain maps $\Phi_*:Q_*\rightarrow \Bar_*(R)$ and $\Psi_*:\Bar_*(R)\rightarrow Q_*$ constructed using the maps $d_t^Q$ and $D_t$ ($t\ge 0$) by the algorithm from \cite{IIVZ} satisfy the equalities $\theta_{\Bar,t}\Phi_t=(\Phi_t)_{(\nu,\nu)}\theta_t$ and $\theta_t\Psi_t=(\Psi_t)_{(\nu,\nu)}\theta_{\Bar,t}$. Then it follows from \cite[Lemma 2]{Volk_new} that $t$-cocycle lies in $\Im\Theta_R^{\nu}$ if and only if it can be represented by a homomorphism $f\in\Hom(Q_t,R)$ such that
\begin{equation}\label{nu_stab}
f(e_{[i_1,j_1][i_2,j_2]})=\nu^{-1}f_{[i_1-1,j_1][i_2-1,j_2]}.
\end{equation}

\section{The description of the Lie bracket and the BV-structure}

In this section we suppose that $R=R(n,r)$.
Let $w$ be a path from a vertex $x$ to a vertex $y$. Denote by $w^*$ the element of
$\Hom(P_{[y][x]},R)$ such that $w^*(e_{y}\otimes e_{x})=w$. For $1\leqslant i\leqslant r$ consider the following auxiliary homomorphisms:
$$
\begin{aligned}
w_{i,m,j}&=(\omega_{i,j+m-1,j})^*\hspace{0.2cm}(0\leqslant m\leqslant n-3, 1\leqslant j\leqslant n-2-m);\\
t_{i,m,j}&=(\mu_{i+1,j+m-(n-1)}\tau_i\eta_{i,j})^*\\
&(1\leqslant m\leqslant n-2, n-1-m\leqslant j\leqslant n-2);\\
u_{i,m,q}&=(\gamma_{i,q}\eta_{i,n-2-m})^*\hspace{0.2cm}(0\leqslant m\leqslant n-3, q\in\{n-1,n\});\\
v_{i,m,q}&=(\mu_{i+1,m}\beta_{i,q})^*\hspace{0.2cm}(0\leqslant m\leqslant n-3, q\in\{n-1,n\});\\
u_{i,q}&={e_{i,q}}^*\hspace{0.2cm}(q\in\{n-1,n\});\\
v_{i,q}&=(\gamma_{i+1,q}\eta_{i+1,1}\beta_{i,q})^*\hspace{0.2cm}(q\in\{n-1,n\}).
\end{aligned}
$$

Let now define some elements of the algebra $\HH^*(R)$.

a) Define 1-cocycle $\ee_1\in\Hom(Q_1,R)$ by the formula
$$\ee_1=u_{1,0,n-1}+u_{1,0,n}.$$

b) Let $s=2m+l(2n-3)$, $0\leqslant m\leqslant n-2$, $r|m+l(n-1)$,
$2|m+ln$ and one of the following conditions is satisfied: $\charr k=2$ or
$2|l$. Define $s$-cocycle $f_s\in\Hom(Q_s,R)$ by the formula
$$
f_s=\sum\limits_{i=1}^r\Big(\sum\limits_{j=1}^{n-2-m}w_{i,m,j}+u_{i,n-1}+ u_{i,n}\Big).
$$

c) Let $s=2m+1+l(2n-3)$, $0\leqslant m\leqslant n-3$, $r|m+l(n-1)$, $2{\not|}m+ln$ and one of the following conditions is satisfied: $\charr k=2$
or $2{\not|}l$. Define $s$-cocycle $g_s\in\Hom(Q_s,R)$ by the formula
$$
g_s=\sum\limits_{i=1}^r \Big(\sum\limits_{j=n-1-m}^{n-2}t_{i,m,j}+
u_{i,m,n-1}+ v_{i,m,n-1}\Big).
$$

d) Let $s=(l+1)(2n-3)-1$, $r|(l+1)(n-1)-1$, $2{\not|}(l+1)n$. Define $s$-cocycle $h_s\in\Hom(Q_s,R)$ by the formula
$$
h_s=\sum\limits_{i=1}^r\sum\limits_{j=1}^{n-2}(-1)^jw_{i,0,j}.
$$

e) Let $s=(l+1)(2n-3)-1$, $r|(l+1)(n-1)-1$, $2|n$ and one of the following conditions is satisfied: $\charr k=2$ or
$2|l$. Define $s$-cocycle $p_s\in\Hom(Q_s,R)$ by the formula
$$
p_s=\sum\limits_{i=1}^r\Big(\sum\limits_{j=1}^{n-2}
(-1)^jw_{i,0,j}+u_{i,n-1}\Big).
$$

f) Let $s=l(2n-3)$, $l\geqslant 1$, $r|l(n-1)-1$, $2|ln$ and either $\charr k=2$ or $2{\not|}l$.
Denote by $\chi_s\in\Hom(Q_s,R)$ the $s$-cocycle, which is equal to $v_{r,n}$ on $P_{[1,n][r,n]}$ and is equal to 0 on other direct summands of $Q_s$.

g) Let $s=l(2n-3)$, $r|l(n-1)-1$, $2{\not|}ln$. Denote by $\xi_s\in\Hom(Q_s,R)$ the $s$-cocycle, which is equal to $t_{r,n-2,1}$ on $P_{[1,1][r,1]}$ and is equal to 0 on other direct summands of $Q_s$.

h) Let $r=1$. For $1\leqslant j\leqslant n-2$ denote by $\ee_0^{(j)}$ the 0-cocycle, which is equal to $t_{1,n-2,j}$ on $P_{[1,j][1,j]}$ and is equal to 0 on other direct summands of $Q_0$. For $q\in\{n-1,n\}$ denote by $\ee_0^{(q)}$ the 0-cocycle, which is equal to $v_{1,q}$ on $P_{[1,q][1,q]}$ and is equal to 0 on other direct summands of $Q_0$.

It was shown in \cite{Volk} that the elements defined in a)--h) are cocycles for the corresponding values of $s$ and that they generate $\HH^*(R)$ as a $k$-algebra. In addition, $\xi_s$ can be excluded from the set of generators in the case where $\charr k{\not|}\frac{n-1}{2}$ and $\chi_s$ can be excluded from the set of generators in the case where $2|n$, $\charr k{\not|}n-1$. Moreover, it is proved in the same work that the elements of the form $f_s$, $g_s$, $h_s$, $p_s$, $\ee_1f_s$, $\ee_1g_s$, $\chi_s$, $\xi_s$ and $\ee_0^{(q)}$ generate $\HH^*(R)$ as a $k$-linear space.

If $\charr k{\not|}r$, then $\HH^*(R)$ is a BV-algebra by the Corollary of Theorem \ref{Tradler_new}. Since $f_s$, $g_s$, $h_s$ and $p_s$ satisfy the condition \eqref{nu_stab}, they lie in the image of $\Theta_R^{\nu}$ (even if $\charr k|r$).

Let us introduce the following notation
$$
F(x)=\begin{cases}
0,&\mbox{ if $x=\ee_1$,}\\
m+l(n-1),&\mbox{ if $x=f_{2m+l(2n-3)}$ or $x=g_{2m+1+l(2n-3)}$,}\\
l(n-1)-1,&\mbox{ if $x=p_{l(2n-3)-1}$, $x=h_{l(2n-3)-1},$}\\
&\mbox{ $x=\xi_{l(2n-3)}$ or $x=\chi_{l(2n-3)}$.}\\
\end{cases}
$$
Note that $r|F(x)$ in all the cases.

\begin{prop}\label{ebrackets} Let $x\in\{\ee_1,f_s,g_s,h_s,p_s\}$. Then
\begin{equation}\label{febracket}
{}[x,\ee_1]=\frac{F(x)}{r}\,x
\end{equation}
in $\HH^*(R)$.
\end{prop}
\begin{proof} Recall the construction of the chain maps $\Phi_t:Q_t\rightarrow\Bar_t(R)$ and $\Psi_t:\Bar_t(R)\rightarrow Q_t$ from \cite{IIVZ}.

Firstly, define $\Phi_0$ and $\Psi_0$ by the equalities $\Phi_0(e_x\otimes e_x)=e_x\otimes e_x$, $\Psi_0(e_x\otimes e_x)=e_x\otimes e_x$, $\Psi_0(e_x\otimes e_y)=0$ for $x,y\in\mathcal Q_0$, $x\not=y$. For $t>0$ the map $\Psi_t$ is defined by the formula
\begin{equation}\label{Psi}
\Psi_t(1\otimes a_1\otimes\dots\otimes a_n\otimes 1)=D_{t-1}(\Psi_{t-1}(1\otimes a_1\otimes\dots\otimes a_{n-1}\otimes 1)a_n).
\end{equation}
From now on we assume that $Q_t=\bigoplus\limits_{p\in X_t} P_{[x_p][y_p]}$ for all $t\ge 0$, where $\{X_t\}_{t\ge 0}$ is a set of disjoint sets. Let $\pi_p:Q_t\rightarrow P_{[x_p][y_p]}$ be the canonical projection and $$\pi_{p'}d_{t-1}^Q(e_{x_p}\otimes e_{y_p})=a_{p,p'}\otimes e_{y_{p'}}+b_{p,p'}\otimes c_{p,p'}$$ for $p'\in X_{t-1}$, $p\in X_t$, where $a_{p,p'}, b_{p,p'}\in R$ and $c_{p,p'}$ are in Jacobson radical of $R$. Then
\begin{equation}\label{Phi}
\Phi_t(e_{x_p}\otimes e_{y_p})=\sum\limits_{p'\in X_{t-1}}b_{p,p'}\Phi_{t-1}(e_{x_{p'}}\otimes e_{y_{p'}}) c_{p,p'}\otimes 1.
\end{equation}
Consequently,
$$\Psi_t\Phi_t(e_{x_p}\otimes e_{y_p})=\sum\limits_{p'\in X_{t-1}}b_{p,p'}D_{t-1}(\Psi_{t-1}\Phi_{t-1}(e_{x_{p'}}\otimes e_{y_{p'}}) c_{p,p'}).$$
It follows from these formulas and induction on $t$ that for all $t\ge 0$, $p\in X_t$ we have
\begin{equation}\label{goodPhitform}
\Phi_t(e_{x_p}\otimes e_{y_p})=e_{x_p}\otimes a_{p,1}\otimes\dots\otimes a_{p,t}\otimes e_{y_p}+\sum\limits_{z\in Y_p}a_{z,0}\otimes\dots\otimes a_{z,t+1},
\end{equation}
where
$$\Psi_t(e_{x_p}\otimes a_{p,1}\otimes\dots\otimes a_{p,t}\otimes e_{y_p})=e_{x_p}\otimes e_{y_p}$$
and $\Psi_t(a_{z,0}\otimes\dots\otimes a_{z,t+1})=0$ for all $z\in Y_p$.

The equality $[\ee_1,\ee_1]=0$ follows from the definition of the Lie bracket and the fact that $\ee_1$ is an element of odd degree. For other elements we use the formula
$$
[f,\ee_1]=((f\Psi_t)\circ(\ee_1\Psi_1)-(\ee_1\Psi_1)\circ(f\Psi_t))\Phi_t.
$$

Consider a $\mathbb{Z}$-grading on the algebra $R$ such that the idempotents and arrows, except $\gamma_{1,n-1}$ and $\gamma_{1,n}$, are of degree 0 and the arrows $\gamma_{1,n-1}$ and $\gamma_{1,n}$ are of degree 1. This grading induces a grading on $\Lambda$.
We can define a grading on the direct summands of $Q_t$ ($t\ge 0$) in such a way that $Q_*$ is a graded resolution of the module $R$. Let $t=t'+l(2n-3)$ ($0\le t'\le 2n-4$), $p\in X_t$ and $P_{[x_p][y_p]}\simeq{}_{\si^l}(P)_1$, where the module $P$ appears in the formula for $Q_{t'}$ as the module $P_{[i_2,j_2][i_1,j_1]}$. We define the degree of the element $e_{x_p}\otimes e_{y_p}\in Q_t$ in the following way:\\
1) if $j_1,j_2\not\in\{n-1,n\}$, then $$\deg (e_{x_p}\otimes e_{y_p})=\left\lceil\frac{i_2+l(n-1)-1}{r}\right\rceil-\left\lceil\frac{i_1-1}{r}\right\rceil;$$
2) if $j_1\not\in\{n-1,n\}$, $j_2\in\{n-1,n\}$, then $$\deg (e_{x_p}\otimes e_{y_p})=\left\lceil\frac{i_2+l(n-1)}{r}\right\rceil-\left\lceil\frac{i_1-1}{r}\right\rceil;$$
3) if $j_1\in\{n-1,n\}$, $j_2{\not\in}\{n-1,n\}$, then $$\deg (e_{x_p}\otimes e_{y_p})=\left\lceil\frac{i_2+l(n-1)-1}{r}\right\rceil-\left\lceil\frac{i_1}{r}\right\rceil;$$
4) if $j_1,j_2\in\{n-1,n\}$, then $$\deg (e_{x_p}\otimes e_{y_p})=\left\lceil\frac{i_2+l(n-1)}{r}\right\rceil-\left\lceil\frac{i_1}{r}\right\rceil.$$
Here we denote by $\lceil a\rceil$ the smallest integer which is greater of equal to $a$. It is easy to show that the differentials $d_t^Q$ are actually of degree 0 for the grading introduced in 1)--4). Moreover, it is easy to check that, if we introduce the grading on the modules $\Bar_t(R)$ in such a way that $\deg(a_0\otimes\dots\otimes a_{t+1})=\sum\limits_{i=0}^{t+1}\deg(a_i)$ for homogeneous elements $a_i\in R$ ($0\le i\le t+1$), then $\Bar_*(R)$ becomes a graded resolution of $R$. In addition $\Phi_t$ ($t\ge 0$) is a homomorphism of graded modules. It is easy to check that $\ee_1\Psi_1(b)=\deg(b)b$ for $b\in B_R$.

It is clear that for elements $f\in C^1(R)=\Homk(R,R)$ and $g\in C^t(R)=\Homk(R^{\otimes t},R)$ the composition product $f\circ g$ is just a composition of $f$ and $g$. Then $\big((\ee_1\Psi_1)\circ(f\Psi_t)\big)\Phi_t=\ee_1\Psi_1f$ for $f\in\Hom(Q_t,R)$, $fd_t^Q=0$. Then it is easy to show that
\begin{equation}\label{ecircf}
\begin{aligned}
&\big((\ee_1\Psi_1)\circ(f\Psi_t)\big)\Phi_t\\
=&\begin{cases}
0,&\mbox{if $f\in\{f_s,h_s,p_s\}$},\\
\sum\limits_{j=n-1-m}^{n-2}t_{1,m,j}+u_{1,m,n-1},&\mbox{if $f=g_s$}.
\end{cases}
\end{aligned}
\end{equation}

It follows from the formula \eqref{goodPhitform} and the formula $(\ee_1\Psi_1)(b)=\deg(b)b$ ($b\in B_R$) that
$$
\big((f\circ\Psi_t)\circ(\ee_1\Psi_1)\big)\Phi_t(e_{x_p}\otimes e_{y_p})=\deg (e_{x_p}\otimes e_{y_p})f(e_{x_p}\otimes e_{y_p})
$$
for $p\in X_t$, $f\in\Hom(Q_t,R)$, $fd_t^Q=0$. The assertion of proposition follows from this formula and \eqref{ecircf}.
\end{proof}

Now we prove a theorem which combined with Proposition \ref{ebrackets} and the results of \cite{Volk} gives a full description of the algebra $\HH^*(R)$ as a Gerstenhaber algebra in all cases and as a BV-algebra in the case $\charr k{\not|}r$.

\begin{theorem}\label{GerBV}
{\rm 1)} If $\charr k{\not|}r$, then $\HH^*(R)$ is a BV-algebra. In this case the BV-differential $\Delta$ is defined by the following equalities:
$$
\begin{aligned}
\Delta(\ee_1)&=\frac{1}{r},\:\:\Delta(f_s)=\Delta(g_s)=\Delta(h_s)=\Delta(p_s)=0,\\
\Delta(\ee_1f_s)&=\frac{f_s}{r}+[f_s,\ee_1],\:\:\Delta(\ee_1g_s)=\frac{g_s}{r}+[g_s,\ee_1],\\
\Delta(\chi_s)&=\frac{l}{r}\left(\frac{n}{2}f_{s-1}-p_{s-1}\right)\hspace{0.2cm}(2|n, s=l(2n-3)),\\
\Delta(\chi_s)&=0\hspace{0.2cm}(2{\not|}n),\:\:
\Delta(\xi_s)=\frac{2l}{r}h_{s-1}\hspace{0.2cm}(s=l(2n-3)),\:\:
\Delta(\ee^{(q)}_0)=0.
\end{aligned}
$$

{\rm 2)} Suppose that $\charr k|r$. Then $\chi_s$ and $\xi_s$ can be excluded from the set of generators and the Lie bracket is defined on the generators of $\HH^*(R)$ by the equalities \eqref{febracket} and the equalities
$$
\begin{aligned}
{}[f_{s_1},f_{s_2}]=&[f_{s_1},g_{s_2}]=[f_{s_1},h_{s_2}]=[f_{s_1},p_{s_2}]=
[g_{s_1},g_{s_2}]=[g_{s_1},h_{s_2}]\\
=&[g_{s_1},p_{s_2}]=[h_{s_1},h_{s_2}]=[p_{s_1},p_{s_2}]=0.
\end{aligned}
$$
\end{theorem}

\begin{proof} 1) In this case $\HH^*(R)$ is a BV-algebra. Let us consider $\tilde\ee_1\in\Hom(Q_1,R)$ defined by the equality $$\tilde\ee_1=\frac{\sum\limits_{i=1}^r(u_{i,0,n-1}+u_{i,0,n})}{r}=\frac{\sum\limits_{i=1}^r\nu^{-i}\ee_1\theta^i}{r}.$$
Then $\tilde\ee_1=\ee_1$ in $\HH^*(R)$ and it is easy to show that $\Delta(\tilde\ee_1\Psi_1)=\frac{1}{r}$.

Let us consider a $\mathbb Z$-grading on the algebra $R$, which is induced by length. This grading induces a grading on $\Lambda$. We can define a grading on the direct summands of $Q_t$ ($t\ge 0$) in such a way that $Q_*$ is a graded resolution of the module $R$. Let $t=t'+l(2n-3)$ ($0\le t'\le 2n-4$), $p\in X_t$ and $P_{[x_p][y_p]}\simeq{}_{\si^l}(P)_1$, where the module $P$ appears in the formula for $Q_{t'}$ as the module $P_{[i_2,j_2][i_1,j_1]}$. Then we define the degree of the element $e_{x_p}\otimes e_{y_p}\in Q_t$ by the formula
$$\deg (e_{x_p}\otimes e_{y_p})=l(n-1)^2+a(n-1)+{\rm min}(j_2,n-1)-{\rm min}(j_1,n-1).$$
It is easy to show that the differentials $d_t^Q$ are actually of degree 0 for this grading. Moreover, it is easy to check that, if we introduce a grading on the modules $\Bar_t(R)$ in such a way that $\deg(a_0\otimes\dots\otimes a_{t+1})=\sum\limits_{i=0}^{t+1}\deg(a_i)$ for homogeneous elements $a_i\in R$ ($0\le i\le t+1$), then $\Bar_*(R)$ becomes a graded resolution of $R$. In addition $\Phi_t$ and $\Psi_t$ ($t\ge 0$) are homomorphisms of graded modules.

Let $M$, $N$ be $\mathbb Z$-graded spaces. We say that a linear map $\ph:M\rightarrow N$ is of degree $q$ and write $\deg\ph=q$ if $\deg\ph(m)=\deg m-q$ for any homogeneous element $m\in M$. Thus it is easy to show that a grading on $R$ induces a grading on $\HH^*(R)$. Then direct inspection shows that
$\deg x=(n-1)F(x)$ for $x\in\{\ee_1,f_s,g_s,h_s,p_s,\xi_s,\chi_s\}$.
Note that any $b\in B_R$ is homogeneous and satisfies the equalities $\deg\bar b=n-1-\deg b$ and $\deg \nu b=\deg b$. In addition we have $\langle a,1\rangle=0$ for homogeneous $a\in R$ such that $\deg a\not=n-1$. Suppose that $x\in \Ker\delta_s^{\nu}$ and $\deg x=q$. Let $a_i\in B_R$ ($1\le i\le n-1$), $A=a_1\otimes\dots\otimes a_{n-1}$. Then
$$
\Delta_ix(A)=\sum\limits_{a\in B_R}\la f(a_i\otimes\dots\otimes a_{n-1}\otimes \bar a\otimes \nu a_1\otimes\dots\otimes \nu a_{i-1}),1\ra a.
$$
Since
$$\deg f(a_i\otimes\dots\otimes a_{n-1}\otimes \bar a\otimes \nu a_1\otimes\dots\otimes \nu a_{i-1})
=\deg A+\deg \bar a-q,$$
nonzero coefficients can appear only for $a\in B_R$ such that $\deg a=\deg A-q$, i.e. $\deg (\Delta_ix)=q$. So $\deg (\Delta x)=\deg x$.

If $x\in\{f_s,g_s,h_s,p_s\}$, then $\deg x=(n-1)F(x)$ and there are no nonzero elements of such degree in $\HH^t(R)$ for $t<s$. So $\Delta(f_s)=\Delta(g_s)=\Delta(h_s)=\Delta(p_s)=0$. Then the equalities for $\Delta(\ee_1f_s)$ and $\Delta(\ee_1g_s)$ follow from \eqref{GerBVeq}. We can calculate $\Delta(\ee_1 p_s)$ and $\Delta(\ee_1 h_s)$ using the same formula. The formula for $\Delta(\chi_s)$ in the case where $2|n$, $\charr k{\not|}n-1$ and the formula for $\Delta(\xi_s)$ in the case where $2{\not|}n$, $\charr k{\not|}\frac{n-1}{2}$ follow from \cite[Lemma 1]{Volk}.

Let now $\tilde\xi_s=\sum\limits_{i=1}^r\nu^{-i}\xi_s\theta^i$, $\tilde\chi_s=\sum\limits_{i=1}^r\nu^{-i}\chi_s\theta^i$. Then $\xi_s=\frac{\tilde\xi_s}{r}$ and $\chi_s=\frac{\tilde\chi_s}{r}$ in $\HH^*(R)$. Note that the elements $\tilde\xi_s\Psi_s$ and $\tilde\chi_s\Psi_s$ belong to $C^s(R)^{\nu}$. In addition, if the elements $\tilde\xi_s\Psi_s$ and $\tilde\chi_s\Psi_s$ are defined for some field $k$, then they are defined for any field (for the same quiver).
As it was said before the formulas for $\Delta(\xi_s)=\frac{\Delta(\tilde\xi_s\Psi_s)\Phi_{s-1}}{r}$ and $\Delta(\chi_s)=\frac{\Delta(\tilde\chi_s\Psi_s)\Phi_{s-1}}{r}$ are valid for a field $k$ with zero characteristic.

Let us introduce the notion of the standard basis for some modules.
The standard basis for $\Lambda$ is the set $B_{\Lambda}=\{a\otimes b\}_{a,b\in B_R}$. If $x,y\in\mathcal Q_0$, then the standard basis for $P_{[x][y]}$ is the set $B_{\Lambda}\cap P_{[x][y]}$. And the standard basis for $R^{\otimes t}$ is the set $\{a_1\otimes\dots\otimes a_t\}_{a_1,\dots,a_t\in B_R}$. Thus we define the standard basis for $Q_t$ ($t\ge 0$) and $\Bar_t(R)$. Note that in all cases the definition of the standard basis does not depend on the field. If the standard basis is defined for a module $M$, we denote it by $B_M$. We denote by $\mathbb ZB_M$ the set of linear combinations with integer coefficients of elements from $B_M$. Note that\\
 a) if $b\in \mathbb ZB_M$ and $a\in B_R$, then $ab, ba\in \mathbb ZB_M$ and the coefficients do not depend on the field;\\
 b) if $a\in \mathbb ZB_{Q_t}$, then $d_{t-1}^Q(a)\in \mathbb ZB_{Q_{t-1}}$, $D_t(a)\in \mathbb ZB_{Q_{t+1}}$ and the coefficients do not depend on the field;\\
 c) if $a\in B_R$, then $\bar a\in B_R$ and $\nu a\in B_R$ and they do not depend on the field;\\
 d) if $a,b\in \mathbb ZB_R$, then $\langle a,b\rangle\in\mathbb Z$ and it does not depend on the field.

 It follows from a) and b) that the elements of the matrices of $\Psi_t$ and $\Phi_t$ in the standard basis are integer and do not depend on the field.
 It follows from c) and d) that if the matrix of $x:R^{\otimes (n+2)}\rightarrow R$ written in the standard basis consists of integer numbers, then the matrix of $\Delta(x):R^{\otimes (n+1)}\rightarrow R$ written in the standard basis consists of integer numbers which do not depend on the field. It follows from our arguments that the matrices of $\Delta(\tilde\xi_s\Psi_s)\Phi_{s-1}:Q_{s-1}\rightarrow R$ and $\Delta(\tilde\chi_s\Psi_s)\Phi_{s-1}:Q_{s-1}\rightarrow R$ written in the standard bases consist of integer numbers which do not depend on the field. If $x\in\{\xi_s, \chi_s\}$, then it follows from \cite[Remark 5]{Volk} that if the set of elements of degree $(n-1)F(x)$ is linear independent in $\Ker(\Hom(d_{s-1}^Q,R))$, then it is linear independent in $\HH^{s-1}(R)$. The formulas for $\Delta(\tilde\xi_s\Psi_s)\Phi_{s-1}$ and $\Delta(\tilde\chi_s\Psi_s)\Phi_{s-1}$ are written in terms of elements of $\Ker(\Hom(d_{s-1}^Q,R))$ whose definitions do not depend on the field. So these formulas remain true for any field.

2) Since the elements $\chi_s$ and $\xi_s$ appear only in the case $\Nod(n-1,r)=1$, they can be excluded from the set of generators in this case. Let $f,g\in\{f_s,g_s,h_s,p_s\}$. Then there are such $x,y\in \HH^*(R)^{\nu\uparrow}$ that $\Theta_R^{\nu}(x)=f$ and $\Theta_R^{\nu}(y)=g$. Since $\Theta_R^{\nu}$ is a homomorphism of Gerstenhaber algebras we have $[f,g]=\Theta_R^{\nu}([x,y])$. Since $\charr k|r$ we have $\charr k\not=2$ or $2|r$.
It is easy to check that in both cases elements $f$ and $g$ have even degree. In addition it follows from the proof of \cite[Lemma 3]{VolkGen} and the formula \eqref{nu_stab} that if $\charr k\not=2$ or $2|r$, then for any $a\in\Ker(\Hom(d_{2s+1}^Q,R))$ such that $a=\nu^{-1}a\theta_{2s+1}$ there is $\bar a\in\Ker(\Hom(d_{2s+1}^Q,R))$ such that $a=\sum\limits_{i=0}^{r-1}\nu^{-i}\bar a\theta_{2s+1}^i$. Since $\bar a^{\nu}=\bar a$ in $\HH^*(R)$ by Corollary 2 of Lemma \ref{chainty} we have
$$
a=\sum\limits_{i=0}^{r-1}\nu^{-i}\bar a\theta_{2s+1}^i=\sum\limits_{i=0}^{r-1}\bar a^{\nu^i}=r\bar a=0
$$
in $\HH^{2s+1}(R)$. Consequently, $\Theta_R^{\nu}(\HH^s(R)^{\nu\uparrow})=0$ for odd $s$. The element $[x,y]$ has odd degree because elements $x$ and $y$ have even degree. Then $[f,g]=\Theta_R^{\nu}([x,y])=0$ and 2) is proved.
\end{proof}

\begin{rema} It is easy to show that we can introduce the BV-structure on $\HH^*(R)$ in the case where $\charr k|r$. For example, we can set $\Delta$ equal to 0 on all generators of $\HH^*(R)$.
\end{rema}

\end{document}